\documentclass[a4paper,twoside,12pt]{amsart}
\usepackage{amsthm, amsmath, amssymb}
\usepackage{xy} \xyoption{all}

\usepackage{HolonomicNoteArXiv}

\title{On a reconstruction theorem for holonomic systems}

\author[A.~D'Agnolo]{Andrea D'Agnolo} 

\address[A.~D'Agnolo]{Dipartimento di Matematica\\ 
Universit{\`a} di Padova\\ 
via Trieste 63, 35121 Padova, Italy} 

\email{dagnolo@math.unipd.it}

\thanks{The first author expresses his gratitude to the RIMS of Kyoto University for hospitality during the preparation of this paper.}

\author[M.~Kashiwara]{Masaki Kashiwara} 

\address[M.~Kashiwara]{Research Institute for Mathematical Sciences\\
Kyoto University\\
Kyoto, 606-8502, Japan}

\email{masaki@kurims.kyoto-u.ac.jp}

\address[M.~Kashiwara]{Department of Mathematical Sciences, Seoul National
University, Seoul, Korea}

\keywords{Riemann-Hilbert problem, holonomic $\D$-modules, ind-sheaves, Stokes phenomenon}

\subjclass[2010]{32C38, 35A20, 32S60, 34M40}

\begin{document}

\maketitle

\begin{abstract}
Let $X$ be a complex manifold.
The classical Riemann-Hilbert correspondence associates to a regular holonomic system $\shm$ the $\C$-constructible complex of its holomorphic solutions.
Let $t$ be the affine coordinate in the complex projective line.
If $\shm$ is not necessarily regular, we associate to it the ind-$\R$-constructible complex $G$ of tempered holomorphic solutions to $\shm\etens\D e^t$.
We conjecture that this provides a Riemann-Hilbert correspondence for holonomic systems.
We discuss the functoriality of this correspondence, we prove that
$\shm$ can be reconstructed from $G$ if $\dim X=1$, and we show how the Stokes data are encoded in $G$.
\end{abstract}

\section*{Introduction}

Let $X$ be a complex manifold. The Riemann-Hilbert correspondence of~\cite{Kas84} establishes an anti-equivalence
\[
\xymatrix{
\BDC_\reghol(\D_X) \ar@<.5ex>[r]^-{\Phi^0} 
& \BDC_{\Cc}(\C_X) \ar@<.5ex>[l]^-{\Psi^0}
}
\]
between regular holonomic $\D$-modules and $\C$-constructible complexes.
Here, $\Phi^0(\shl) = \rhom[\D_X](\shl,\OO_X)$ is the complex of holomorphic solutions to $\shl$,
and $\Psi^0(L) = \thom(L,\OO_X) = \rhom(L,\OOt_X)$ is the complex of holomorphic functions tempered along $L$.
Since $\shl\simeq\Psi^0(\Phi^0(\shl))$, this shows in particular that $\shl$ can be reconstructed from $\Phi^0(\shl)$.

We are interested here in holonomic $\D$-modules which are not necessarily regular.

The theory of ind-sheaves from~\cite{KS01} allows one to consider the complex $\Phi^{\mathsf t}(\shm) = \rhom[\D_X](\shm,\OOt_X)$ of tempered holomorphic solutions to a holonomic module $\shm$.
The basic example $\Phi^{\mathsf t}(\D_\C e^{1/x})$ was computed in~\cite{KS03}, and 
the functor $\Phi^{\mathsf t}$ has been studied in~\cite{Mor07,Mor10}. However, since $\Phi^{\mathsf t}(\D_\C e^{1/x}) \simeq \Phi^{\mathsf t}(\D_\C e^{2/x})$, one cannot reconstruct $\shm$ from $\Phi^{\mathsf t}(\shm)$.

Set $\Phi(\shm) = \Phi^{\mathsf t}(\shm \etens \D_\PP e^t)$, for $t$ the affine variable in the complex projective line $\PP$. This is an ind-$\R$-constructible complex in $X\times\PP$.
The arguments in~\cite{DAg12} suggested us how $\shm$ could be reconstructed from $\Phi(\shm)$ via a functor $\Psi$, described below (\S\ref{se:corr}).

We conjecture that the contravariant functors
\[
\xymatrix{
\BDC(\D_X) \ar@<.5ex>[r]^-\Phi 
& \BDC(\mathrm{I}\C_\XP) \ar@<.5ex>[l]^-\Psi,
}
\]
between the derived categories of $\D_X$-modules and of ind-sheaves on $\XP$,
provide a Riemann-Hilbert correspondence for holonomic systems.

To corroborate this statement, we discuss the functoriality of $\Phi$ and $\Psi$ with respect to proper direct images and to tensor products with regular objects (\S\ref{se:funct}). This allows a reduction to holonomic modules with a good formal structure.

When $X$ is a curve and $\shm$ is holonomic, we prove that the natural morphism $\shm\to\Psi(\Phi(\shm))$ is an isomorphism (\S\ref{se:reconstruction}).
Thus $\shm$ can be reconstructed from $\Phi(\shm)$.

Recall that irregular holonomic modules are subjected to the Stokes phenomenon. We describe with an example how the Stokes data of $\shm$ are encoded topologically in the ind-$\R$-constructible sheaf $\Phi(\shm)$ (\S\ref{se:Stokes}).

In this Note, the proofs are only sketched.
Details will appear in a forthcoming paper.
There, we will also describe some of the properties of 
the essential image of holonomic systems by the functor $\Phi$.
Such a category is related to a construction of~\cite{Tam10}.

\section{Notations}

We refer to \cite{KS90,KS01,Kas03}.

Let $X$ be a real analytic manifold.

Denote by $\BDC(\C_X)$ the bounded derived category of sheaves of $\C$-vector spaces, and by $\BDC_{\Rc}(\C_X)$ the full subcategory of objects with $\R$-constructible cohomologies. Denote by $\tens$, $\rhom$, $\opb f$, $\roim f$, $\reim f$, $\epb f$ the six Grothendieck operations for sheaves. (Here $f\colon X\to Y$ is a morphism of real analytic manifolds.)

For $S\subset X$ a locally closed subset,  we denote by $\C_S$ the zero extension to $X$ of the constant sheaf on $S$.

Recall that an ind-sheaf is an ind-object in the category of sheaves with compact support.
Denote by $\BDC(\mathrm{I}\C_X)$ the bounded derived category of ind-sheaves, and by $\BDC_{\mathrm{I}\Rc}(\mathrm{I}\C_X)$ the full subcategory of objects with ind-$\R$-constructible cohomologies. Denote by $\tens$, $\rihom$, $\opb f$, $\roim f$, $\reeim f$, $\epb f$ the six Grothendieck operations for ind-sheaves.

Denote by $\alpha$ the left adjoint of the embedding of sheaves into ind-sheaves. One has $\alpha(\smash{\indlim} F_i) = \smash{\ilim} F_i$. Denote by $\beta$ the left adjoint of $\alpha$.

Denote by $\Dbt_X$ the ind-$\R$-constructible sheaf of tempered distributions. 

Let $X$ be a complex manifold.
We set for short $d_X = \dim X$.

Denote by $\OO_X$ and $\D_X$ the rings of holomorphic functions and of differential operators. Denote by $\Omega_X$ the invertible sheaf of differential forms of top degree.

Denote by $\BDC(\D_X)$ the bounded derived category of left $\D_X$-modules, and by $\BDC_{\hol}(\D_X)$ and $\BDC_{\reghol}(\D_X)$ the full subcategories of objects with holonomic and regular holonomic cohomologies, respectively. Denote by $\dtens$, $\dopb f$, $\doim f$ the operations for $\D$-modules.  (Here $f\colon X\to Y$ is a morphism of complex manifolds.)

Denote by $\ddual\shm$ the dual of $\shm$ (with shift such that $\ddual\OO_X\simeq\OO_X$).

For $Z\subset X$ a closed analytic subset, we denote by $\rsect_{[Z]}\shm$ and $\shm(*Z)$ the relative algebraic cohomologies of a $\D_X$-module $\shm$.

Denote by $\ssv(\shm)\subset X$ the singular support of $\shm$, that is the set of points where the characteristic variety is not reduced to the zero-section.

Denote by $\OOt_X\in\BDC_{\mathrm{I}\Rc}(\mathrm{I}\C_X)$ the complex of tempered holomorphic functions. Recall that $\OOt_X$ is the Dolbeault complex of $\Dbt_X$ and that it has a structure of $\beta\D_X$-module. We will write for short $\rhom[\D_X](\shm,\OOt_X)$ instead of $\rihom[\beta\D_X](\beta\shm,\OOt_X)$.

\section{Exponential $\D$-modules}

Let $X$ be a complex analytic manifold.
Let $D\subset X$ be a hypersurface, and set $U=X\setminus D$.
For $\varphi\in\OO_X(*D)$, we set
\begin{align*}
\D_X e^\varphi &= \D_X/\{P\colon Pe^\varphi=0 \text{ on } U\}, \\
\she^\varphi_{D|X}&=(\D_X\, e^\varphi)(*D).
\end{align*}
As an $\OO_X(*D)$-module, $\she^\varphi_{D|X}$ is generated by $e^\varphi$.
Note that $\ssv(\she^\varphi_{D|X}) = D$, and $\she^\varphi_{D|X}$ is holonomic. It is regular if $\varphi\in\sho_X$, since then
$\she^\varphi_{D|X} \simeq \OO_X(*D)$. 

One easily checks that
$(\ddual\she^\varphi_{D|X})(*D) \simeq \she^{-\varphi}_{D|X}$.

\begin{proposition}
\label{pro:dualphi}
If $\dim X =1$, and $\varphi$ has an effective pole at every point of $D$, then
$\ddual\she^\varphi_{D|X} \simeq \she^{-\varphi}_{D|X}$.
\end{proposition}

Let $\PP$ be the complex projective line and
denote by $t$ the coordinate on $\C=\PP\setminus\{\infty\}$.

For $c\in\R$, we set for short
\begin{align*}
\{\RE \varphi <c\} =& \{x\in U\colon \RE \varphi(x) <c\}, \\
\{\RE(t+\varphi)<c\} =& \{(x,t) \colon x\in U,\ t\in\C, \RE(t+\varphi(x))<c\}.
\end{align*}
Consider the ind-$\R$-constructible sheaves on $X$ and on $\XP$, respectively,
\begin{align*}
\C_{\{\RE \varphi <?\}} &= \indlim[c\to+\infty]\C_{\{\RE \varphi <c\}}, \\
\C_{\{\RE(t+\varphi) <?\}} &= \indlim[c\to+\infty]\C_{\{\RE(t+\varphi) <c\}}.
\end{align*}

The following result is analogous to \cite[Proposition~7.1]{DAg12}.
Its proof is simpler than loc.\ cit., since $\varphi$ is differentiable.

\begin{proposition}
\label{pro:Ephi}
One has an isomorphism in $\BDC(\D_X)$
\[
\she^\varphi_{D|X} \isoto \roim q \rhom[\opb p\D_\PP](\opb p\she^t_{\infty|\PP}, 
\rhom(\C_{\{\RE(t+\varphi)<?\}},\OOt_\XP)),
\]
for $q$ and $p$ the projections from $\XP$.
\end{proposition}

The following result is analogous to \cite[Proposition~7.3]{KS03}.

\begin{lemma}
\label{lem:s/t}
Denote by $(u,v)$ the coordinates in $\C^2$.
There is an isomorphism in $\BDC(\mathrm{I}\C_{\C^2})$
\[
\rhom[\D_{\C^2}](\she^{u/v}_{\{v=0\}|\C^2},\OOt_{\C^2}) \simeq
\rihom(\C_{\{v\neq 0\}},\C_{\{\RE u/v < ?\}}).
\]
\end{lemma}

\begin{proposition}
\label{pro:Solphi}
There is an isomorphism in $\BDC(\mathrm{I}\C_X)$
\[
\rhom[\D_X](\ddual\she^{-\varphi}_{D|X},\OOt_X) \simeq \rihom(\C_U,\C_{\{\RE \varphi < ?\}}).
\]
\end{proposition}

\begin{proof}
As $\ddual\she^{u/v}_{\{v=0\}|\C^2} \simeq \she^{-u/v}_{\{v=0\}|\C^2}$, Lemma~\ref{lem:s/t} gives
\[
\OOvt_{\C^2} \ltens[\D_{\C^2}] \she^{-u/v}_{\{v=0\}|\C^2}[-2] \simeq  \rihom(\C_{\{v\neq 0\}},\C_{\{\RE u/v < ?\}}).
\]
Write $\varphi=a/b$ for $a,b\in\OO_X$ such that $b^{-1}(0) \subset D$, and consider the map
\[
f = (a,b) \colon X\to \C^2.
\]
As $\dopb f\she^{-u/v}_{\{v=0\}|\C^2} \simeq \she^{-\varphi}_{D|X}$, \cite[Theorem~7.4.1]{KS01} implies
\[
\OOvt_X \ltens[\D_X] \she^{-\varphi}_{D|X} [-d_X]
\simeq \rihom(\C_U,\C_{\{\RE \varphi < ?\}}).
\]
Finally, one has
\[
\OOvt_X \ltens[\D_X] \she^{-\varphi}_{D|X} [-d_X]
\simeq
\rhom[\D_X](\ddual\she^{-\varphi}_{D|X},\OOt_X).
\]
\end{proof}

\section{A correspondence}\label{se:corr}

Let $X$ be a complex analytic manifold.
Recall that $\PP$ denotes the complex projective line.
Consider the contravariant functors
\[
\xymatrix{
\BDC(\D_X) \ar@<.5ex>[r]^-\Phi 
& \BDC(\mathrm{I}\C_\XP) \ar@<.5ex>[l]^-\Psi
}
\]
defined by
\begin{align*}
\Phi(\shm) &= \rhom[\D_\XP](\shm\detens\she^t_{\infty|\PP},\OOt_\XP), \\
\Psi(F) &= \roim q \rhom[\opb p\D_\PP](\opb p\she^t_{\infty|\PP}, \rhom(F,\OOt_\XP)),
\end{align*} 
for $q$ and $p$ the projections from $\XP$.

We conjecture that this provides a Riemann-Hilbert correspondence for holonomic systems:

\begin{conjecture}
\label{con:main}
\begin{itemize}
\item [(i)]
The natural morphism of endofunctors of $\BDC(\D_X)$
\begin{equation}
\label{eq:IdPsiPhi}
\id \to\Psi \circ \Phi
\end{equation}
is an isomorphism on $\BDC_\hol(\D_X)$.
\item [(ii)]
The restriction of $\Phi$
\[
\Phi|_{\BDC_\hol(\D_X)} \colon \BDC_\hol(\D_X) \to \BDC(\mathrm{I}\C_\XP)
\]
is fully faithful.
\end{itemize}
\end{conjecture}

Let us prove some results in this direction.

\section{Functorial properties}\label{se:funct}

The next two Propositions are easily deduced from the results in~\cite{KS01}.

\begin{proposition}
\label{pro:proper}
Let $f\colon X\to Y$ be a proper map, and set $f_\PP = f\times\id_\PP$. 
Let $\shm \in\BDC_\hol(\D_X)$ and $F\in\BDC_{\mathrm{I}\Rc}(\mathrm{I}\C_\XP)$.
Then
\begin{align*}
\Phi( \doim f\shm ) & \simeq \reeim{f_\PP} \Phi(\shm) [d_X-d_Y],\\
\Psi( \reeim {f_\PP} F ) & \simeq \doim f \Psi(F) [d_X-d_Y].
\end{align*}
\end{proposition}

For $\shl\in\BDC_\reghol(\D_X)$, set
\[
\Phi^0(\shl) = \rhom[\D_X](\shl,\OO_X).
\]
Recall that $\Phi^0(\shl)$ is a $\C$-constructible complex of sheaves on $X$.

\begin{proposition}
\label{pro:tens}
Let $\shl\in\BDC_\reghol(\D_X)$, $\shm \in\BDC_\hol(\D_X)$ and $F\in\BDC_{\mathrm{I}\Rc}(\mathrm{I}\C_\XP)$.
Then
\begin{align*}
\Phi( \ddual(\shl\dtens\ddual\shm) ) & \simeq \rihom(\opb q \Phi^0(\shl),\Phi(\shm)) ,\\
\Psi( F \tens \opb q \Phi^0(\shl) ) & \simeq \Psi(F) \dtens \shl.
\end{align*}
\end{proposition}

Noticing that
\[
\Phi(\OO_X) \simeq \C_X \etens \rihom(\C_{\{t\neq\infty\}},\C_{\{\RE t<?\}}),
\]
one checks easily that $\Psi(\Phi(\OO_X)) \simeq \OO_X$. Hence, Proposition~\ref{pro:tens} shows:

\begin{theorem}
\begin{itemize}
\item [(i)]
For $\shl\in\BDC_\reghol(\D_X)$, we have
\begin{align*}
\Phi(\shl) 
&\simeq \opb q \Phi^0(\shl) \tens \Phi(\OO_X) \\
&\simeq \Phi^0(\shl) \etens \rihom(\C_{\{t\neq\infty\}},\C_{\{\RE t<?\}}).
\end{align*}
\item [(ii)]
The morphism \eqref{eq:IdPsiPhi} is an isomorphism on $\BDC_\reghol(\D_X)$.
\item [(iii)]
For any $\shl,\shl'\in\BDC_\reghol(\D_X)$, the natural morphism
\[
\Hom[\D_X](\shl,\shl') \to \Hom(\Phi(\shl'),\Phi(\shl))
\]
is an isomorphism.
\end{itemize}
\end{theorem}

Therefore, Conjecture~\ref{con:main} holds true for regular holonomic $\D$-modules.

\section{Review on good formal structures}

Let $D\subset X$ be a hypersurface. A flat meromorphic connection with poles at $D$ is a holonomic $\D_X$-module $\shm$ such that $\ssv(\shm) = D$ and $\shm \simeq \shm(*D)$.

We recall here the classical results on the formal structure of flat meromorphic connections on curves. (Analogous results in higher dimension have been obtained in~\cite{Moc11,Ked11}.)

\smallskip

Let $X$ be an open disc in $\C$ centered at $0$.

For $\shf$ an $\OO_X$-module, we set
\[
\shf{\widehat |}_0 = \widehat\OO_{X,0} \tens[\OO_{X,0}]\shf_0, 
\]
where $\widehat\OO_{X,0}$ is the completion of $\OO_{X,0}$.

One says that a flat meromorphic connection $\shm$ with poles at $0$ has a good formal structure if
\begin{equation}
\label{eq:fM}
\shm{\widehat |}_0 \simeq 
\DSum_{i\in I}\left(\shl_i\dtens\she^{\varphi_i}_{0|X}\right){\widehat |}_0
\end{equation}
as $(\widehat\OO_{X,0} \tens[\OO_{X,0}]\D_{X,0})$-modules, where $I$ is a finite set, $\shl_i$ are regular holonomic $\D_X$-modules, and $\varphi_i\in\OO_X(*0)$.

A ramification at $0$ is a map $X\to X$ of the form $x\mapsto x^m$ for some $m\in\N$.

The Levelt-Turrittin theorem asserts:

\begin{theorem}
\label{thm:LT}
Let $\shm$ be a meromorphic connection with poles at $0$.
Then there is a ramification $f\colon X\to X$ such that $\dopb f\shm$ 
has a good formal structure at $0$.
\end{theorem}

Assume that $\shm$ satisfies \eqref{eq:fM}. If $\shm$ is regular, then $\varphi_i\in\OO_X$ for all $i\in I$, and \eqref{eq:fM} is induced by an isomorphism
\[
\shm_0 \simeq 
\DSum_{i\in I}\left(\shl_i\dtens\she^{\varphi_i}_{0|X}\right)_0.
\]
However, such an isomorphism does not hold in general.

Consider the real oriented blow-up
\begin{equation}
\label{eq:blow}
\pi\colon B = \R\times S^1 \to X,
\quad (\rho,\theta) \mapsto \rho e^{i\theta}.
\end{equation}
Set $V = \{\rho>0\}$ and let $Y = \{\rho \geq 0\}$ be its closure. 
If $W$ is an open neighborhood of $(0,\theta)\in\partial Y$, then $\pi(W\cap V)$ contains a germ of open sector around the direction $\theta$ centered at $0$.

Consider the commutative ring
\[
\sha_Y = \rhom[\opb\pi\D_{\overline X}](\opb\pi\OO_{\overline X}, \rhom(\C_V,\Dbt_B)),
\]
where $\overline X$ is the complex conjugate of $X$.

To a $\D_X$-module $\shm$, one associates the $\sha_Y$-module
\[
\pi^*\shm = \sha_Y \tens[\opb\pi\OO_X]\opb\pi\shm.
\]

The Hukuara-Turrittin theorem states that \eqref{eq:fM} can be extended to germs of open sectors:

\begin{theorem}
\label{thm:blow}
Let $\shm$ be a flat meromorphic connection with poles at $0$.
Assume that $\shm$ admits the good formal structure \eqref{eq:fM}.
Then for any $(0,\theta)\in \partial Y$ one has
\begin{equation}
\label{eq:atheta}
(\pi^*\shm)_{(0,\theta)} \simeq
\bigl(\DSum_{i\in
I}\pi^*\bigl(\she^{\varphi_i}_{0|X}\bigr)^{m_i}\bigr)_{(0,\theta)},
\end{equation}
where $m_i$ is the rank of $\shl_i$.
\end{theorem}

(Note that only the ranks of the $\shl_i$'s appear here, since $x^\lambda(\log x)^m$ belongs to $\sha_Y$ for any $\lambda\in\C$ and $m\in\Z_{\geq 0}$.)

One should be careful that the above isomorphism depends on $\theta$, giving rise to the Stokes phenomenon.

We will need the following result:

\begin{lemma}
\label{lem:pipi}
If $\shm$ is a flat meromorphic connection with poles at $0$, then
\[
\roim\pi(\pi^*\shm) \simeq \shm.
\]
\end{lemma}

\section{Reconstruction theorem on curves}
\label{se:reconstruction}

Let $X$ be a complex curve.
Then Conjecture~\ref{con:main}~(i) holds true:

\begin{theorem}
\label{thm:main}
For $\shm\in\BDC_\hol(\D_X)$ there is a functorial isomorphism
\begin{equation}
\label{eq:main}
\shm \isoto \Psi(\Phi(\shm)).
\end{equation}
\end{theorem}

\begin{proof}[Sketch of proof]
Since the statement is local, we can assume that $X$ is an open disc in $\C$ centered at $0$, and that $\ssv(\shm) = \{0\}$.

By devissage, we can assume from the beginning that $\shm$ is a flat meromorphic connection with poles at $0$.

Let $f\colon X\to X$ be a ramification as in Theorem~\ref{thm:LT},
so that $\dopb f\shm$ admits a good formal structure at $0$.

Note that $\doim f \dopb f \shm \simeq \shm \dsum \shn$ for some $\shn$.
If \eqref{eq:main} holds for $\dopb f\shm$, then it holds for $\shm\dsum\shn$ by Proposition~\ref{pro:proper}, and hence it also holds for $\shm$. 

We can thus assume that $\shm$ admits a good formal structure at $0$.

Consider the real oriented blow-up \eqref{eq:blow}. 

By Lemma~\ref{lem:pipi}, one has $\shm\simeq\roim\pi\pi^*\shm$.
Hence Proposition~\ref{pro:proper} (or better, its analogue for $\pi$) 
implies that we can replace $\shm$ with $\pi^*\shm$.

By Theorem~\ref{thm:blow}, we finally reduce to prove
\[
\she_{0|X}^\varphi \isoto 
\Psi(\Phi(\she_{0|X}^\varphi)).
\]

Set $D' = \{x=0\}\cup\{t=\infty\}$ and $U'=(\XP)\setminus D'$.
By Proposition~\ref{pro:dualphi},
\[
\ddual \she_{D'|\XP}^{t+\varphi} \simeq
\ddual(\she_{0|X}^\varphi\detens\she_{\infty|\PP}^t)
\simeq \she_{D'|\XP}^{-t-\varphi}.
\]

By Proposition~\ref{pro:Solphi}, we thus have
\[
\Phi(\she_{0|X}^\varphi) 
\simeq \rihom(\C_{U'},\C_{\{\RE(t+\varphi)<?\}}).
\]

Noticing that $\Phi(\she_{0|X}^\varphi)\tens\C_{D'} \in \BDC_\Cc(\C_\XP)$, one checks that $\Psi(\Phi(\she_{0|X}^\varphi)\tens\C_{D'}) \simeq 0$.

Hence, Proposition~\ref{pro:Ephi} implies
\[
\Psi(\Phi(\she_{0|X}^\varphi) )
\simeq \Psi(\C_{\{\RE(t+\varphi)<?\}}) 
\simeq \she_{0|\C}^\varphi.
\]
\end{proof}

\begin{example}
Let $X = \C$, $\varphi(x) = 1/x$ and $\shm = \she^\varphi_{0|X}$.
Then we have
\[
H^k\Phi(\shm) = 
\begin{cases}
\C_{\{\RE(t+\varphi) <?\}}, & \text{for }k=0, \\
\C_{\{x= 0,\ t\neq\infty\}} \dsum \C_{\{x\neq 0,\ t=\infty\}}, & \text{for }k=1, \\
0, & \text{otherwise}.
\end{cases}
\]
\end{example}

\section{Stokes phenomenon}\label{se:Stokes}

We discuss here an example which shows how, in our setting, the Stokes phenomenon arises in a purely topological fashion.

\medskip

Let $X$ be an open disc in $\C$ centered at $0$.
(We will shrink $X$ if necessary.) 
Set $U = X\setminus \{0\}$.

Let $\shm$ be a flat meromorphic connection with poles at $0$ such that
\[
\shm{\widehat |}_0 \simeq 
(\she^{\varphi}_{0|X} \dsum \she^{\psi}_{0|X}){\widehat |}_0,
\quad \varphi,\psi\in\OO_X(*0).
\]
Assume that $\psi-\varphi$ has an effective pole at $0$.

The Stokes curves of $\she^{\varphi}_{0|X} \dsum \she^{\psi}_{0|X}$ are the real analytic arcs $\ell_i$, $i\in I$, defined by
\[
\{\RE(\psi-\phi) = 0\} = \bigsqcup\nolimits_{i\in I}\ell_i.
\]
(Here we possibly shrink $X$ to avoid crossings of the $\ell_i$'s and to ensure that they admit the polar coordinate $\rho>0$ as parameter.)

Since $\she^{\varphi}_{0|X} \simeq \she^{\varphi+\varphi_0}_{0|X}$ for $\varphi_0\in\OO_X$, the Stokes curves are not invariant by isomorphism.

The Stokes lines $L_i$, defined as the limit tangent half-lines to $\ell_i$ at $0$, are invariant by isomorphism.

The Stokes matrices of $\shm$ describe how the isomorphism \eqref{eq:atheta} changes when $\theta$ crosses a Stokes line.

Let us show how these data are topologically encoded in $\Phi(\shm)$.

Set $D' = \{x=0\}\cup\{t=\infty\}$ and $U'=(\XP)\setminus D'$.
Set
\begin{align*}
F_c &= \C_{\{\RE(t+\varphi)<c\}}, &
G_c &= \C_{\{\RE(t+\psi)<c\}}, \\
F &= \C_{\{\RE(t+\varphi)<?\}} , &
G &= \C_{\{\RE(t+\psi)<?\}} .
\end{align*}
By Proposition~\ref{pro:Solphi} and Theorem~\ref{thm:blow},
\[
\Phi(\shm) \simeq \rihom(\C_{U'},H),
\]
where $H$ is an ind-sheaf such that
\[
H\tens\C_{\opb q S} \simeq (F\dsum G)\tens\C_{\opb q S}
\] 
for any sufficiently small open sector $S$.

Let $\fb^\pm$ be the vector space of upper/lower triangular matrices in $\operatorname{M}_2(\C)$, and let $\ft=\fb^+\cap \fb^-$ be the vector space of diagonal matrices.

\begin{lemma}\label{lem:FcGc}
Let $S$ be an open sector, and $\mathfrak{v}$ a vector space, which satisfy one of the following conditions:
\begin{itemize}
\item [(i)]
$\mathfrak{v} = \fb^\pm$ and $S\subset \{\pm\RE(\psi-\varphi)>0\}$,
\item [(ii)]
$\mathfrak{v}=\ft$, $S\supset L_i$ for some $i\in I$ and $S\cap L_j=\varnothing$ for $i\neq j$.
\end{itemize}
Then, for $c'\gg c$, one has
\[
\Hom((F_c\dsum G_c)|_{\opb q S}, (F_{c'}\dsum G_{c'})|_{\opb q S}) \simeq
\mathfrak{v}.
\]
In particular,
\[
\Endo((F\dsum G)\tens\C_{\opb q S}) \simeq \mathfrak{v}.
\]
\end{lemma}

This proves that the Stokes lines are encoded in $H$.
Let us show how to recover the Stokes matrices of $\shm$ as glueing data for $H$.

Let $S_i$ be an open sector which contains $L_i$ and is disjoint from $L_j$ for $i\neq j$. We choose $S_i$ so that $\Union\nolimits_{i\in I} S_i = U$.

Then for each $i\in I$, there is an isomorphism
\[
\alpha_i\colon H\tens\C_{\opb q S_i} \simeq (F\dsum G)\tens\C_{\opb q S_i}.
\]

Take a cyclic ordering of $I$ such that the Stokes lines get ordered counterclockwise. 

Since $\{S_i\}_{i\in I}$ is an open cover of $U$, the ind-sheaf $H$ is reconstructed from $F\dsum G$ via the glueing data given by the Stokes matrices
\[
A_i = \alpha_{i+1}^{-1}\alpha_i|_{\opb q(S_i\cap S_{i+1})}\in \fb^\pm.
\]

\providecommand{\bysame}{\leavevmode\hbox to3em{\hrulefill}\thinspace}


\begin{thebibliography}{10}

\bibitem{DAg12} A. D'Agnolo,
\emph{On the Laplace transform for temperate holomorphic functions},
\texttt{arXiv:1207.5278} (2012), 30 pp.


\bibitem{DK11} A. D'Agnolo and M. Kashiwara,
\emph{On quantization of complex symplectic manifolds},
Commun. Math. Phys. \textbf{308} (2011), no. 1, 81--113

\bibitem{Kas84} M. Kashiwara,
\emph{The Riemann-Hilbert problem for holonomic systems}, 
Publ. Res. Inst. Math. Sci. \textbf{20} (1984), no. 2, 319--365.

\bibitem{Kas03} M. Kashiwara,
$\D$-modules and Microlocal Calculus,
Translations of Mathematical Monographs \textbf{217}, American Math. Soc. (2003), xvi+254 pp.

\bibitem{KS90} M. Kashiwara and P. Schapira, 
Sheaves on manifolds, 
Grundlehren der Mathematischen Wissenschaften \textbf{292}, Springer-Verlag (1990), x+512 pp.

\bibitem{KS96} M. Kashiwara and P. Schapira, 
Moderate and formal cohomology associated with constructible sheaves,
M\'em. Soc. Math. France \textbf{64} (1996), iv+76 pp. 

\bibitem{KS01} M. Kashiwara and P. Schapira, 
Ind-sheaves,
Ast\'erisque \textbf{271} (2001), 136 pp.

\bibitem{KS03} M. Kashiwara and P. Schapira, 
\emph{Microlocal study of ind-sheaves. I. Micro-support and regularity}, Ast\'erisque \textbf{284} (2003), 143--164.

\bibitem{Ked11} K.S.~Kedlaya,
\emph{Good formal structures for flat meromorphic connections,  II: Excellent schemes},
J. Amer. Math. Soc. {\bf 24} (2011), 183--229.

\bibitem{Moc11} T.~Mochizuki,
Wild Harmonic Bundles and Wild Pure Twistor $\D$-modules,
Ast\'erisque \textbf{340} (2011), x+607 pp.

\bibitem{Mor07} G. Morando,
\emph{An existence theorem for tempered solutions of $\D$-modules on complex curves},
Publ. Res. Inst. Math. Sci. \textbf{43} (2007),  no. 3, 625--659.

\bibitem{Mor10} G. Morando,
\emph{Preconstructibility of tempered solutions of holonomic $\D$-modules},
\texttt{arXiv:1007.4158} (2010), 24 pp.

\bibitem{Sab00} C.~Sabbah,
\emph{\'Equations diff\'erentielles \`a points singuliers irr\'eguliers et ph{\'e}nom{\`e}ne de Stokes en dimension 2},
Ast{\'e}risque \textbf{263} (2000), viii+190 pp.

\bibitem{Tam10} D. Tamarkin, 
\emph{Microlocal condition for non-displaceablility}, \texttt{arXiv:1003.3304} (2010), 93 pp.

\end{thebibliography}
\end{document}